\documentclass{elsarticle}

\usepackage[all,cmtip]{xy}
\usepackage{array}
\usepackage{amsmath, amssymb, amsthm}
\usepackage{graphics}
\usepackage{caption,subcaption}
\usepackage{algorithm, algorithmic}
\usepackage{booktabs}
\usepackage{color}

\def\rmd{{\rm d}}
\def\deg{{\rm deg}}

\newtheorem{theorem}{{\bf Theorem}}

\newtheorem{lemma}[theorem]{{\bf Lemma}}
\newtheorem{example}{{\bf Example}}

\def\bfa{{\boldsymbol{a}}}
\def\bfb{{\boldsymbol{b}}}

\def\bfp{{\boldsymbol{p}}}
\def\bfx{{\boldsymbol{x}}}

\def\mB{\boldsymbol{B}}
\def\mC{\boldsymbol{C}}
\def\mI{\boldsymbol{I}}
\def\mQ{\boldsymbol{Q}}
\def\mM{\boldsymbol{M}}

\def\tbfx{\tilde{\bfx}}
\def\ta{\tilde{a}}
\def\tb{\tilde{b}}

\def\RR{\mathbb{R}}

\def\AAA{\mathcal{A}}
\def\CCC{\mathcal{C}}

\def\GGG{\mathcal{G}}
\def\LLL{\mathcal{L}}
\def\OOO{\mathcal{O}}
\def\tOOO{\tilde{\mathcal{O}}}

\def\Res{\mathrm{Res}}

\begin{document}

\begin{frontmatter}
\title{Symmetry Detection of Rational Space Curves\\ from their Curvature and Torsion}


\author[a]{Juan Gerardo Alc\'azar\fnref{proy}}
\ead{juange.alcazar@uah.es}
\author[a]{Carlos Hermoso}
\ead{carlos.hermoso@uah.es}
\author[b]{Georg Muntingh}
\ead{georgmu@math.uio.no}

\address[a]{Departamento de F\'{\i}sica y Matem\'aticas, Universidad de Alcal\'a,
E-28871 Madrid, Spain}
\address[b]{SINTEF ICT, PO Box 124 Blindern, 0314 Oslo, Norway}

\fntext[proy]{Supported by the Spanish ``Ministerio de
Ciencia e Innovacion" under the Project MTM2011-25816-C02-01. Partially supported by
the Jos\'e Castillejos' grant $\mbox{CAS12$/$00022}$ from the Spanish Ministerio de Educaci\'on, Cultura y Deporte.
Member of the Research Group {\sc asynacs} (Ref. {\sc ccee2011/r34}) }


\begin{abstract}
We present a novel, deterministic, and efficient method to detect whether a given rational space curve is symmetric. By using well-known differential invariants of space curves, namely the curvature and torsion, the method is significantly faster, simpler, and more general than an earlier method addressing a similar problem~\cite{AHM13-2}. To support this claim, we present an analysis of the arithmetic complexity of the algorithm and timings from an implementation in {\tt Sage}.
\end{abstract}

\end{frontmatter}

\section{Introduction}\label{section-introduction}
\noindent The problem of detecting the symmetries of curves and surfaces has attracted the attention of many researchers throughout the years, because of the interest from fields like Pattern Recognition \cite{Boutin, Calabi, Huang, LR08, LRTh, Suk1, Suk2, TC00, Taubin2, Weiss}, Computer Graphics \cite{Berner08, Bokeloh, Lipman, Martinet,  Mitra06, Podolak, Schnabel, Simari}, and Computer Vision \cite{Alt88, Brass, Jiang, Li08, Li10, Loy, Tate, Sun}. The introduction in \cite{AHM13-2} contains an extensive account of the variety of approaches used in the above references.

A common characteristic in most of these papers is that the methods focus on computing \emph{approximate} symmetries more than exact symmetries, which is perfectly reasonable in many applications, where curves and surfaces often serve as merely approximate representations of a more complex shape. Some exceptions appear here: If the object to be considered is discrete (e.g. a polyhedron), or is described by a discrete object, like for instance a control polygon or a control polyhedron, then the symmetries can be determined exactly \cite{Alt88, Brass, Jiang, Li08}. Examples of the second class are B\'ezier curves and tensor product surfaces. Furthermore, in these cases the symmetries of the curve or surface follow from those of the underlying discrete object. Another exception appears in \cite{LR08}, where the authors provide a deterministic method to detect rotation symmetry of an implicitly defined algebraic plane curve and to find the exact rotation angle and rotation center. The method uses a complex representation of the curve and is generalized in \cite{LRTh} to detect mirror symmetry as well.

Rational curves are frequently used in Computer Aided Geometric Design and are the building blocks of NURBS curves. Compared to implicit curves, rational parametric curves are easier to manipulate and visualize. Space curves appear in a natural way when intersecting two surfaces, and they play an important role when dealing with special types of surfaces, often used in geometric modeling, like ruled surfaces, canal surfaces or surfaces of revolution, which are generated from a {\it directrix} or {\it profile} curve. Furthermore, in geometric modeling it is typical to use rational space curves as profile curves.

In this paper we address the problem of deterministically finding the symmetries of a rational space curve, defined by means of a proper parametrization. Notice that since we deal with a global object, i.e., the set of all points in the image of a rational parametrization, and not just a piece of it, the discrete approach from \cite{Alt88, Brass, Jiang, Li08} is not suitable here. Determining if a rational space curve is symmetric or not is useful in order to properly describe the topology of the curve \cite{AD10}. Furthermore, if the space curve is to be used for generating, for instance, a canal surface or a surface of revolution, certain symmetries of the curve will be inherited by the generated surface. Hence, for modeling purposes it can be interesting to know these symmetries in advance.

Recently, the problem of determining whether a rational plane or space curve is symmetric has been addressed in \cite{A13, AHM13, AHM13-2} using a different approach. The common denominator in these papers is the following observation: If a rational curve is \emph{symmetric}, i.e., invariant under a nontrivial isometry $f$, then this symmetry induces another parametrization of the curve, different from the original parametrization. Assuming that the initial parametrization is proper (definition below), the second parametrization is also proper. Since two proper parametrizations of the same curve are related by a M\"obius transformation \cite{SWPD}, determining the symmetries is reduced to finding this transformation, therefore translating the problem to the parameter space. This observation leads to algorithms for determining the symmetries of plane curves with polynomial parametrizations \cite{A13} and of plane and space curves with rational parametrizations \cite{AHM13-2}, although in the latter case of general space curves only involutions were considered. The more general problem of determining whether two rational plane curves are similar was considered in \cite{AHM13}.

In this paper we again employ the above observation, but in addition we now also use well-known differential invariants of space curves, namely the curvature and the torsion. The improvement over the method in \cite{AHM13-2} is threefold: First of all, we are now able to find \emph{all} the symmetries of the curve instead of just the involutions. Secondly, the new algorithm is considerably faster and can efficiently handle even curves with high degrees and large coefficients in reasonable timings. Finally, the method is simpler to implement and requires fewer assumptions on the parametrization.

Some general facts on symmetries of rational curves are presented in Section~\ref{sec-prelim}. Section~\ref{sec-sym-detec} provides an algorithm for checking whether a curve is symmetric. The determination of the symmetries themselves is addressed in Section~\ref{sec-find}. Finally, in Section~\ref{exp-sec} we report on the performance of the algorithm, by presenting a complexity analysis and providing timings for several examples, including a comparison with the curves tested in \cite{AHM13-2}.

\section{Symmetries of rational curves} \label{sec-prelim}
\noindent Throughout the paper, we consider a rational space curve $\CCC\subset \RR^3$, neither a line nor a circle, parametrized by a rational map
\begin{equation}\label{eq:parametrizations}
{\bfx}: \RR \dashrightarrow \CCC\subset \RR^3, \qquad \bfx(t) = \big( x(t),y(t),z(t) \big) .
\end{equation}
The components $x(t), y(t), z(t)$ of $\bfx$ are rational functions of $t$ with rational coefficients, and they are defined for all but a finite number of values of $t$. Let the \emph{(parametric)} degree $m$ of $\bfx$ be the maximal degree of the numerators and denominators of the components $x(t),y(t),z(t)$. Note that rational curves are irreducible. We assume that the parametrization \eqref{eq:parametrizations} is \emph{proper}, i.e., birational or, equivalently, injective except for perhaps finitely many values of $t$. This can be assumed without loss of generality, since any rational curve can quickly be properly reparametrized. For these claims and other results on properness, the interested reader can consult \cite{SWPD} for plane curves and \cite[\S 3.1]{A12} for space curves. 

We recall some facts from Euclidean geometry \cite{Coxeter}. An \emph{isometry} of $\RR^3$ is a map $f:\RR^3\longrightarrow \RR^3$ preserving Euclidean distances. Any isometry $f$ of $\RR^3$ is linear affine, taking the form
\begin{equation}\label{eq:isometry}
f(\bfx) = \mQ\bfx + \bfb, \qquad \bfx\in \RR^3,
\end{equation}
with $\bfb \in \RR^3$ and $\mQ\in \RR^{3\times 3}$ an orthogonal matrix. In particular $\det(\mQ) = \pm 1$. Under composition, the isometries of $\RR^3$ form the \emph{Euclidean group}, which is generated by \emph{reflections}, i.e., symmetries with respect to a plane, or \emph{mirror symmetries}. An isometry is called \emph{direct} when it preserves the orientation, and \emph{opposite} when it does not. In the former case $\det(\mQ) = 1$, while in the latter case $\det(\mQ) = -1$. The identity map of $\RR^3$ is called the \emph{trivial symmetry}.

The classification of the nontrivial isometries of Euclidean space includes reflections (in a plane), rotations (about an axis), and translations, and these combine in commutative pairs to form twists, glide reflections, and rotatory reflections. Composing three reflections in mutually perpendicular planes through a point $\bfp$, yields a \emph{central inversion} (also called \emph{central symmetry}), with center $\bfp$, i.e., a symmetry with respect to the point $\bfp$. The particular case of rotation by an angle $\pi$ is of special interest, and it is called a \emph{half-turn}. Rotation symmetries are direct, while mirror and central symmetries are opposite.

\begin{lemma}
A rational space curve $\CCC\subset \RR^3$ different from a line cannot be invariant under a translation, glide reflection, or twist.
\end{lemma}
\begin{proof}
If $\CCC$ were invariant under translation by a vector $\bfb$, then, for any point $\bfx$ on $\CCC$, the line $\LLL = \{\bfx + t\bfb\,:\,t\in \RR\}$ would intersect $\CCC$ in infinitely many points, implying that $\LLL\subset \CCC$ and contradicting that $\CCC$ is an irreducible curve different from a line. Since applying a glide reflection twice yields a translation, $\CCC$ cannot be invariant under a glide reflection either.
Suppose $\CCC$ is invariant under a twist $f$ with axis $\AAA$ and angle $\alpha$, and let $\pi: \RR^3 \longrightarrow \Pi$ be the orthogonal projection onto a plane $\Pi\perp \AAA$. Then the projection $\CCC' := \pi(\CCC)$ is a plane algebraic curve invariant under the rotation $\pi \circ f$ by the angle $\alpha$ about the point $\AAA\cap \Pi$. By Lemma 1 in \cite{AHM13-2}, $\alpha=2\pi/k$ with $k\leq \deg(\CCC')$. But then $\CCC$ is invariant under the translation $f^k$, which is a contradiction.
\end{proof}

Therefore, the rotations, reflections, and their combinations (like central inversions) are the only isometries leaving an irreducible algebraic space curve, different from a line, invariant. We say that an irreducible algebraic space curve is \emph{symmetric}, if it is invariant under one of these (nontrivial) isometries. In that case, we distinguish between a \emph{mirror symmetry}, \emph{rotation symmetry} and \emph{central symmetry}. If the curve is neither a line nor a circle, it has a finite number of symmetries~\cite{AHM13-2}.

We recall the following result from \cite{AHM13-2}. For this purpose, let us recall first that a \emph{M\"obius transformation} (of the affine real line) is a rational function
\begin{equation}\label{eq:Moebius}
\varphi: \RR\dashrightarrow \RR,\qquad \varphi(t) = \frac{a t + b}{c t + d},\qquad \Delta := ad - bc \neq 0.
\end{equation}
In particular, we refer to $\varphi(t) = t$ as the \emph{trivial transformation}. It is well known that the birational functions on the real line are the M\"obius transformations~\cite{SWPD}.

\begin{theorem}\label{th-fund}
Let $\bfx: \RR\dashrightarrow \CCC\subset \RR^3$ be a proper parametric curve as in \eqref{eq:parametrizations}. The curve $\CCC$ is symmetric if and only if there exists a nontrivial isometry $f$ and nontrivial M\"obius transformation $\varphi$ for which we have a commutative diagram
\begin{equation}\label{eq:fundamentaldiagram}
\xymatrix{
\CCC \ar[r]^{f} & \CCC \\
\RR \ar@{-->}[u]^{\bfx} \ar@{-->}[r]_{\varphi} & \RR \ar@{-->}[u]_{\bfx}
}
\end{equation}
Moreover, for each isometry $f$ there exists a unique M\"obius transformation $\varphi$ that makes this diagram commute.
\end{theorem}

Note that $\varphi(t)$ is the parameter value corresponding to the image under the symmetry $f$ of the point on $\CCC$ with parameter $t$.

\begin{lemma}
Let $\varphi$ be a M\"obius transformation associated to a parametrization~$\bfx$ and isometry $f$ in the sense of Theorem \ref{th-fund}. Then its coefficients $a,b,c,d$ can be assumed to be real, by dividing by a common complex number if necessary.
\end{lemma}
\begin{proof}
For any proper parametrization $\bfx$ and isometry $f$ the associated M\"obius transformation $\varphi = \bfx^{-1}\circ f \circ \bfx$ maps the real line to itself. In particular, since
\[
0 = \varphi(t) - \overline{\varphi(t)}
  = \frac{(a\overline{c} - \overline{a}c)t^2 + (b\overline{c} - \overline{b}c + a\overline{d} - \overline{a}d)t + (b\overline{d} - \overline{b}d)}{(ct + d)(\overline{c}t + \overline{d})}
\]
for any $t$ for which $\varphi(t)$ is defined, $a\overline{c}$ and $b\overline{d}$ are real, so that $\arg(a) = \arg(c)$ and $\arg(b) = \arg(d)$. A similar argument for $\varphi^{-1}$ yields that
$-d\overline{c}/|ad - bc|^2$ and $-b\overline{a}/|ad - bc|^2$ are real, implying that $\arg(c) = \arg(d)$ and $\arg(a) = \arg(b)$. It follows that all coefficients of $\varphi$ have a common argument $\theta$. Therefore, after dividing the coefficients of $\varphi$ by $\exp(i\theta)$, the coefficients of $\varphi$ can be assumed to be real.
\end{proof}

Let the \emph{curvature} $\kappa$ and \emph{torsion} $\tau$ of a parametric curve $\bfx$ be the functions
\[\kappa = \kappa_\bfx := \frac{\Vert \bfx' \times \bfx''\Vert}{\Vert \bfx' \Vert^3}, \qquad
    \tau =   \tau_\bfx := \frac{\langle \bfx'\times \bfx'', \bfx'''\rangle}{\Vert \bfx' \times \bfx''\Vert^2}\]
of the parameter $t$. Note that $\kappa$ is non-negative. The functions $\kappa^2$ and $\tau^2$ are well-known rational \emph{differential invariants} of the parametrization $\bfx$, in the sense that
\begin{equation}\label{eq:kappatauinv}
\kappa_{f\circ \bfx} = \kappa_\bfx, \qquad \tau_{f\circ \bfx} = \det(\mQ)\cdot \tau_\bfx
\end{equation}
for any isometry $f(\bfx) = \mQ\bfx + \bfb$. This follows immediately from $\mQ$ being orthogonal and the identity
\begin{equation}\label{eq:cross}
(\mM\bfa) \times (\mM\bfb) = \det(\mM) \mM^{-\mathrm{T}} (\bfa \times \bfb),
\end{equation}
which holds for any invertible matrix $\mM$ and follows from a straightforward calculation. Although $\tau_\bfx$ and $\kappa_\bfx^2$ are rational for any rational parametrization~$\bfx$, the curvature $\kappa_\bfx$ is in general not rational.

The following lemma describes the behavior of the curvature and torsion under reparametrization, for instance by a M\"obius transformation.

\begin{lemma} \label{Moeb-curvat}
Let $\bfx$ be the rational parametrization \eqref{eq:parametrizations} and let $\phi \in C^3(U)$, with $U\subset \RR$ open. Then
\[ \kappa_{\bfx \circ \phi} = \kappa_\bfx \circ \phi,\qquad \tau_{\bfx \circ \phi} = \tau_\bfx\circ\phi,\]
whenever both sides are defined.
\end{lemma}

\begin{proof}
Writing $\tilde{\bfx} := \bfx\circ \phi$ and using the chain rule, one finds
\begin{align*}
\tilde{\bfx}'(t)
 & = \bfx'  \big(\phi(t)\big)\cdot \phi'(t),\\
\tilde{\bfx}''(t)
 & = \bfx'' \big(\phi(t)\big)\cdot \big(\phi'(t)\big)^2 +\bfx'\big(\phi(t)\big)\cdot \phi''(t),\\
\tilde{\bfx}'''(t)
 & = \bfx'''\big(\phi(t)\big)\cdot \big(\phi'(t)\big)^3 +3\bfx''\big(\phi(t)\big)\cdot \phi'(t)\cdot \phi''(t)+\bfx'\big(\phi(t)\big)\cdot \phi'''(t),
\end{align*}
whenever $t\in U$ and $\bfx$ is defined at $\phi(t)$. Therefore
\[ \kappa_{\bfx \circ \phi}(t)
 = \frac{\left\| \tilde{\bfx}'(t) \times \tilde{\bfx}''(t) \right\|}{\left\| \tilde{\bfx}'(t) \right\|^3}
 = \frac{\left\| \bfx'\big(\phi(t)\big) \times \bfx''\big(\phi(t)\big)\right\|\cdot \left| \phi'(t)\right|^3}{\left\| \bfx'\big(\phi(t)\big) \right\|^3 \cdot \left| \phi'(t)\right|^3}
 = \big( \kappa_\bfx \circ \phi\big) (t), \]
and similarly one finds $\tau_{\bfx\circ \phi} = \tau_\bfx \circ \phi$.
\end{proof}

\section{Symmetry detection} \label{sec-sym-detec}
\noindent In this section we derive a criterion for the presence of nontrivial symmetries $f(\bfx) = \mQ\bfx + \bfb$ of curves of type \eqref{eq:parametrizations}, together with an efficient method for checking this criterion. The cases $\det(\mQ) = \pm 1$ need to be checked separately, but are considered simultaneously using linked $\pm$ and $\mp$ signs consistently throughout the paper. The resulting method is summarized in Algorithm {\tt Symm$^\pm$}.

\subsection{A criterion for the presence of symmetries}\label{subsec-main-th}
\noindent For any parametric curve $\bfx$ as in \eqref{eq:parametrizations}, write
\[\kappa_\bfx^2(t) =: \frac{A(t)}{B(t)}, \qquad \tau_\bfx(t)=:\frac{C(t)}{D(t)},\]
with $(A,B)$ and $(C,D)$ pairs of coprime polynomials. Let
\begin{equation}\label{eq:G}
G^\pm_\bfx := \gcd\big(K_\bfx, T^\pm_\bfx\big),
\end{equation}
with
\begin{equation}\label{eq:KT}
K_\bfx (t, s) := A(t)B(s) - A(s)B(t),\qquad T^\pm_\bfx (t, s) := C(t)D(s) \mp C(s)D(t)
\end{equation}
the result of clearing denominators in the equations
\begin{equation}\label{eq:kappatauts}
\kappa_\bfx^2(t) - \kappa_\bfx^2(s) = 0,\qquad \tau_\bfx (t) \mp \tau_\bfx (s) = 0.\ \ \ \
\end{equation}
Similarly, associate to any M\"obius transformation $\varphi$
the \emph{M\"obius-like} polynomial
\begin{equation}\label{eq:F}
F(t, s) := (ct + d)s - (at + b),\qquad ad - bc\neq 0,
\end{equation}
as the result of clearing denominators in $s - \varphi(t) = 0$. We call $F$ \emph{trivial} when $F(t, s) = s - t$, i.e., when the associated M\"obius transformation is the identity. Note that $F$ is irreducible since $ad - bc\neq 0$. 

\begin{theorem} \label{funda-sym}
Consider the curve $\CCC$ defined by $\bfx$ in \eqref{eq:parametrizations} and let $G^\pm_\bfx$ be as above. Then $\CCC$ has a nontrivial symmetry $f(\bfx) = \mQ\bfx + \bfb$, with $\det(\mQ) = \pm 1$, if and only if there exists a nontrivial polynomial $F$ of type \eqref{eq:F}, associated with a M\"obius transformation $\varphi$, such that $F$ divides $G^\pm_\bfx$ and the parametrizations $\bfx$ and $\bfx\circ \varphi$ have identical speed,
\begin{equation}\label{new-cond}
\|\bfx'\| = \|(\bfx \circ \varphi)'\|.
\end{equation}
\end{theorem}

The zeroset of $F$ is the graph of $\varphi$, which is either a rectangular hyperbola with horizontal and vertical asymptotes when $c\neq 0$, or a line with nonzero and finite slope $a/d$ when $c = 0$.
Whenever $F$ is a factor of $G^\pm_\bfx$, the corresponding hyperbola or line is contained in the zeroset of $G^\pm_\bfx$; see Figure~\ref{fig:Example1}.

\begin{proof}[Proof of Theorem \ref{funda-sym}] ``$\Longrightarrow$'': If $\CCC$ is invariant under a nontrivial isometry $f(\bfx)=\mQ\bfx+\bfb$, with $\det(\mQ) = \pm 1$, by Theorem~\ref{th-fund} there exists a M\"obius transformation $\varphi$ such that $f \circ \bfx = \bfx \circ \varphi$. Let $F$ be the M\"obius-like polynomial associated with $\varphi$. The points $(t,s)$ for which $K_\bfx(t, s) = T^\pm_\bfx(t,s) = 0$ are the points satisfying $\kappa_\bfx(s) = \kappa_\bfx(t)$ and $\tau_\bfx(s) = \pm \tau_\bfx(t)$. This includes the zeroset $\left\{\big(t, s\big)\,:\,s = \varphi(t)\right\}$ of $F(t,s)$, since
\[
\kappa_\bfx \circ \varphi = \kappa_{\bfx\circ \varphi} = \kappa_{f\circ \bfx} = \kappa_\bfx,\qquad
\tau_\bfx \circ \varphi = \tau_{\bfx \circ \varphi} = \tau_{f\circ \bfx} = \det(\mQ) \tau_\bfx = \pm \tau_\bfx
\]
by Lemma \ref{Moeb-curvat} and \eqref{eq:kappatauinv}. Since $F$ is irreducible, B\'ezout's theorem implies that $F$ divides $K_\bfx$ and $T^\pm_\bfx$, and therefore $G^\pm_\bfx$ as well. Furthermore, since $\mQ$ is orthogonal, the parametrizations have equal speed,
\[ \|(\bfx \circ \varphi)'\|  = \|(f\circ \bfx)'\| = \|(\mQ \bfx + \bfb)'\|  = \|\mQ \bfx'\| = \|\bfx'\|. \]

``$\Longleftarrow$'': Let $\varphi$ be the nontrivial transformation associated to $F$. Let $t_0 \in I\subset \RR$ be such that $\bfx(t)$ is a regular point on $\CCC$ for every $t\in I$, and consider the arc length function
\[ s = s(t) := \int_{t_0}^t \|\bfx'(t)\|\rmd t,\qquad t\in I,\]
which (locally) has an infinitely differentiable inverse $t = t(s)$. By \eqref{new-cond},
\[ \left\| \frac{\rmd }{\rmd s} \big(\bfx\circ t\big)\right\|
 = \left\| \frac{\rmd \bfx}{\rmd t} \frac{\rmd t}{\rmd s} \right\|
 = 1
 = \left\| \frac{\rmd }{\rmd t} \big(\bfx\circ \varphi \big) \frac{\rmd t}{\rmd s} \right\|
 = \left\| \frac{\rmd }{\rmd s} \big(\bfx\circ \varphi \circ t\big)\right\|
 ,\]
so that $\bfx\circ t$ and $\bfx\circ \varphi \circ t$ are parametrized by arc length. Since $F$ divides $G^\pm_\bfx$, any zero $\big(t,\varphi(t)\big)$ of $F$ is also a zero of $K_\bfx$ and $T^\pm_\bfx$, implying that $\kappa_\bfx = \kappa_{\bfx} \circ \varphi$ and $\tau_\bfx = \pm \tau_{\bfx} \circ \varphi$. Then, by repeatedly applying Lemma \ref{Moeb-curvat},
\begin{equation}
\kappa_{\bfx \circ t} = \kappa_{\bfx} \circ t = \kappa_{\bfx\circ \varphi} \circ t = \kappa_{\bfx\circ \varphi\circ t}, \quad
\tau_{\bfx \circ t} =
\tau_{\bfx} \circ t =
\pm \tau_{\bfx \circ \varphi} \circ t =
\pm \tau_{\bfx \circ \varphi \circ t}.
\end{equation}
The Fundamental Theorem of Space Curves \cite[p. 19]{Docarmo} then implies that $\bfx\circ t$ and $\bfx\circ \varphi\circ t$ coincide on $s(I)$ up to an isometry $f(\bfx) = \mQ\bfx + \bfb$ with $\det(\mQ)=\pm 1$. Therefore $\CCC$ and $f(\CCC)$ have infinitely many points in common. Since $\CCC$ and $f(\CCC)$ are irreducible algebraic curves, it follows that $\CCC=f(\CCC)$ and therefore $f$ is a symmetry of $\CCC$.
\end{proof}

Note that the polynomial $G^\pm_\bfx$ cannot be identically 0. Indeed, $G^\pm_\bfx$ is identically $0$ if and only if $K_\bfx$ and $T^\pm_\bfx$ are both identically $0$, which happens precisely when $\kappa_\bfx$ and $\tau_\bfx$ are both constant. If $\kappa_\bfx = 0$ then $\CCC$ is a line, if $\tau_\bfx = 0$ and $\kappa_\bfx$ is a nonzero constant then $\CCC$ is a circle, and if $\kappa_\bfx, \tau_\bfx$ are both constant but nonzero then $\CCC$ is a circular helix, which is non-algebraic. All of these cases are excluded by hypothesis.

\subsection{Finding the M\"obius-like factors $F$ of $G^\pm_\bfx$} \label{Check-cond}
\noindent The criterion in Theorem \ref{funda-sym} requires to check if a bivariate polynomial $G = G^\pm_\bfx$ has \emph{real} factors of the form $F(t,s)=(ct+d)s-(at+b)$, with $ad-bc\neq 0$. However, $a,b,c,d$ need not be rational numbers, so that we need to factor over the algebraic or real numbers. This problem has been studied by several authors \cite{C04, CG06, CGKW, GR02}. However, since in our case we are looking for factors of a specific form, we develop an ad hoc method to check the condition.

Let $\GGG$ be the curve in the $(t,s)$-plane defined by $G(t,s)$. Let $t_0$ be such that the vertical line $\LLL$ at $t = t_0$ does not contain any zero of $\GGG$ where the partial derivative $G_s := \frac{\partial G}{\partial s}$ vanishes; see Figure \ref{fig:Example1}. These are the points $t_0$ for which the discriminant of $g(s) := G(t_0, s)$ does not vanish, which is up to a factor equal to the Sylvester resultant $\Res_s(G, G_s)$ and has degree at most $(2m_s - 1)m_t$ in $t_0$, with $(m_t, m_s)$ the bidegree of $G$. Therefore one can always find an integer abscissa $t_0$ with this property by checking for at most $(2m_s - 1)m_t + 1$ points $t_0$ whether the gcd of $g(s)$ and $G_s(t_0, s)$ is trivial.

If $G$ has a M\"obius-like factor $F$ as in \eqref{eq:F}, then the zeroset of $F$ intersects $\LLL$ in a single point $\bfp_0=(t_0, \xi)$ satisfying
\begin{equation} \label{valuesabc-1}
(ct_0 + d)\xi - (at_0 + b) = 0.
\end{equation}
Since $G_s(\bfp_0)\neq 0$, the equation $F(t,s) = 0$ implicitly defines a function $s = s(t)$ in a neighborhood of $\bfp_0$. Moreover, by differentiating the identity $F\big(t,s(t)\big) = 0$ once and twice with respect to $t$, and evaluating at $\bfp_0$, we find the relations
\begin{align}
-a + ds'_0 +c\cdot (\xi + t_0 s'_0)  & = 0,\label{valuesabc-2a}\\
d s''_0 + c\cdot (2s'_0 + t_0 s''_0) & = 0,\label{valuesabc-2b}
\end{align}
where $\xi = s(t_0)$, $s'_0 := s'(t_0)$, and $s''_0 := s''(t_0)$. In order to find expressions for $s'_0$, $s''_0$, we now use that the function $s(t)$ is also implicitly defined by $G(t, s) = 0$, because $F$ is a factor of $G$ and $G_s(\bfp_0)\neq 0$. Differentiating once and twice the identity $G\big(t,s(t)\big) = 0$ with respect to $t$ gives
\begin{equation} \label{eq:difffunc}
s'  = -\frac{G_t(t,s)}{G_s(t,s)},\qquad
s'' = -\frac{G_{tt}(t,s) +2G_{ts}(t,s)s' + G_{ss}(t,s)\big(s'\big)^2}{G_s(t,s)}.
\end{equation}
Evaluating these expressions at $\bfp_0$ yields expressions $s'_0 = s'_0 (\xi)$ and $s''_0 = s''_0 (\xi)$.

Now we distinguish the cases $d\neq 0$ and $d=0$. In the first case, we may assume $d=1$ by dividing all coefficients in the M\"obius transformation by $d$. In that case $2s_0' + t_0 s_0'' = 2\Delta/(ct_0 + 1)^3 \neq 0$ and \eqref{valuesabc-1}--\eqref{valuesabc-2b} yield rational expressions
\begin{equation}\label{eq:abcd=1}
c_1(\xi) := \frac{-s''_0}{2s'_0+t_0 s''_0},\ a_1(\xi) := s'_0 + c_1(\xi) (\xi + t_0 s'_0),\ b_1(\xi) := -a_1(\xi) t_0 + \xi + c_1(\xi) t_0\xi.
\end{equation}
The polynomial $F$ is a factor of $G$ if and only if the resultant $\Res_s(F,G)$ is identically $0$. Substituting $a_1(\xi), b_1(\xi), c_1(\xi)$, and $d=1$ into this resultant yields a polynomial $P_1(t)$, whose coefficients are rational functions of $\xi$. Let $R_1(\xi)$ be the gcd of the numerators of these coefficients and of $g(\xi)$. The real roots $\xi$ of $R_1(\xi)$ for which $a_1(\xi), b_1(\xi), c_1(\xi)$ are well defined and $\Delta_1 (\xi) := a_1(\xi) - b_1(\xi) c_1(\xi)\neq 0$ correspond to the M\"obius-like factors $F$ of $G$ as in \eqref{eq:F} with $d = 1$.

On the other hand, when $d=0$ we may assume $c=1$, and \eqref{valuesabc-1}--\eqref{valuesabc-2b} yield rational expressions
\begin{equation}\label{eq:ab}
a_0(\xi) := \xi + t_0 s'_0,\qquad b_0(\xi) := -a_0(\xi) t_0 + t_0\xi.
\end{equation}
Substituting $a_0(\xi), b_0(\xi), c = 1$, and $d = 0$ into the resultant $\Res_s(F,G)$ yields a polynomial $P_0(t)$, whose coefficients are rational functions of $\xi$. Let $R_0(\xi)$ be the gcd of the numerators of these coefficients and $g(\xi)$. The real roots $\xi$ of $R_0(\xi)$ for which $a_0(\xi)$ and $b_0(\xi)$ are well defined and $\Delta_0(\xi) := -b_0(\xi)$ is nonzero correspond to the M\"obius-like factors $F$ of $G$ as in \eqref{eq:F} with $d = 0$. We obtain the following theorem.
\begin{theorem} \label{factor}
The polynomial $G$ has a real M\"obius-like factor $F$ as in \eqref{eq:F} with $d\neq 0$ (resp. $d = 0$) if and only if $R_1(\xi)$ (resp. $R_0(\xi)$) has a real root. Furthermore, every such real root provides a factor of this form.
\end{theorem}

Note that the cases $d = 0$ and $d \neq 0$ can be computed in parallel.

\begin{figure}
\begin{center}
\includegraphics[scale=0.8]{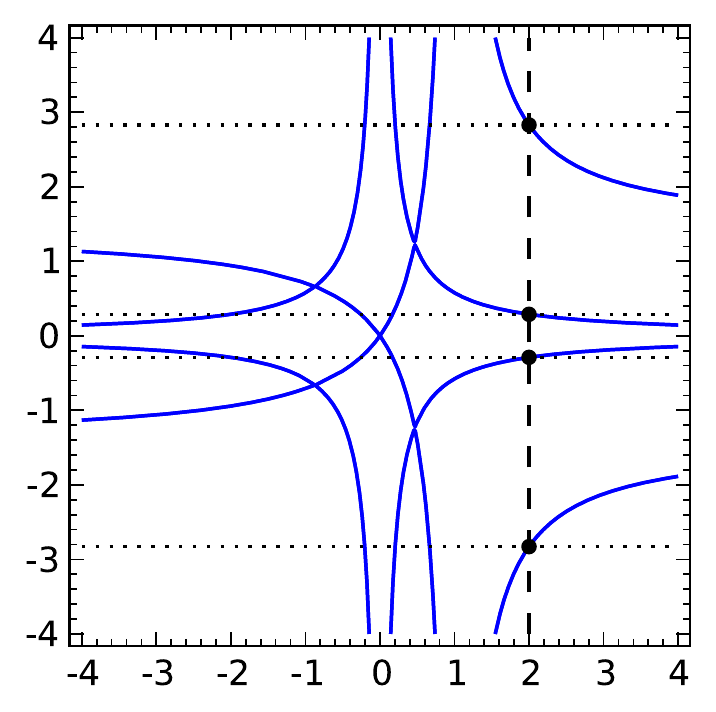}
\end{center}
\caption{The zeroset (solid) of the polynomial $G=G^\pm_\bfx$ intersects the vertical line $\LLL$ (dashed) in the points $(2, \pm \sqrt{8})$ and $(2, \pm 1/\sqrt{12})$ in Example \ref{ex1}.}\label{fig:Example1}
\end{figure}

\begin{example} \label{ex1}
Consider the bivariate polynomial
\[ G(t,s) = 3s^4t^4 - 6s^4t^3 + 3s^4t^2 - 6s^2t^4 - s^2t^2 + 2s^2t - s^2 + 2t^2. \]
The vertical line $\LLL := \{t = t_0 := 2\}$ does not intersect the zeroset of $G$ in a point where $G_s$ vanishes, since the discriminant of $g(\xi) := G(t_0, \xi) = 12\xi^4 - 97\xi^2 + 8$ is nonzero (see Figure \ref{fig:Example1}). Evaluating \eqref{eq:difffunc} at $\bfp_0 = (2,\xi)$ yields
\begin{align*}
s'_0  & = -\frac{18\xi^4-97\xi^2+4}{\xi(24\xi^2-97)}, \\
s''_0 & =  \frac{16416\xi^{10}-206316\xi^8+879669\xi^6-1387682\xi^4+55302\xi^2+1552}{\xi^3(24\xi^2-97)^3}.
\end{align*}
When $d\neq 0$, we may assume $d=1$ and Equations \eqref{eq:abcd=1} yield
\begin{align*}
c_1(\xi) & = -\frac{1}{2}\frac{16416\xi^{10}-206316\xi^8+879669\xi^6-1387682\xi^4+55302\xi^2+1552}{6048\xi^{10}-66636\xi^8+256371\xi^6-456385\xi^4+17666\xi^2+1552},\\
a_1(\xi) & = -\frac{1}{2}\frac{\xi(72\xi^8+2019\xi^6-21192\xi^4+40138\xi^2-4656)}{504\xi^8-5511\xi^6+20905\xi^4-36290\xi^2-1552},\\
b_1(\xi) & = -\frac{(9504\xi^{10}-163836\xi^8+879621\xi^6-1434145\xi^4+133646\xi^2-4656)\xi}{(504\xi^8-5511\xi^6+20905\xi^4-36290\xi^2-1552)(12\xi^2-1)}.
\end{align*}
Substituting these expressions into the resultant $\Res_s (F, G)$ and taking the gcd of the numerators of its coefficients and $g$ yields a polynomial $R_1(\xi)=\xi^2 - 8$. We find $F_1(t,s) = -st+\sqrt{2}t+s$ for $\xi = \sqrt{8}$ and $F_2(t,s)=-st-\sqrt{2}t+s$ for $\xi=-\sqrt{8}$ as factors of $G$. In the case $d=0$, we may assume $c=1$ and we get
\[ a_0(\xi) = - \frac{(\xi^2 - 8)(12\xi^2 - 1)}{\xi(24\xi^2 - 97)}, \qquad b_0(\xi) = 4 \frac{18\xi^4 - 97\xi^2 + 4}{ \xi(24\xi^2 - 97)}. \]
Here $R_0(\xi)=12\xi^2-1$ and we obtain $F_3(t,s)=st-\frac{1}{3}\sqrt{3}$ for $\xi=1/\sqrt{12}$ and $F_4(t,s)=st+\frac{1}{3}\sqrt{3}$ for $\xi=-1/\sqrt{12}$. The entire computation takes a fraction of a second when implemented in {\tt Sage} on a modern laptop. For more details we refer to the worksheet accompanying this paper \cite{WebsiteGeorg}.
\end{example}

\begin{algorithm}[h!]
\begin{algorithmic}[1]
\REQUIRE A proper parametrization $\bfx$ of a space curve $\CCC$, not a line or a circle.
\ENSURE The number of symmetries $f(\bfx) = \mQ\bfx + \bfb$, with $\det(\mQ)=\pm 1$, of $\CCC$.
\STATE Find the bivariate polynomials $K, T^\pm$, and $G^\pm$ from \eqref{eq:G} and \eqref{eq:KT}.
\STATE Find the resultant $\Res_s(F, G^\pm)$, with $F$ as in \eqref{eq:F}.
\STATE Let $t_0$ be such that the discriminant of $g^\pm (\xi) := G^\pm(t_0, \xi)$ does not vanish.
\STATE Find the gcd $R_1(\xi)$ of $g^\pm$ and the numerators of the coefficients of the polynomial $P_1(t)$ obtained by substituting $d = 1$ and \eqref{eq:abcd=1} into $\Res_s(F, G^\pm)$.
\STATE Find the real roots of $R_1(\xi)$ for which \eqref{eq:kappatauts} is well defined, each defining a M\"obius transformation by substituting \eqref{eq:abcd=1} and $d = 1$ in \eqref{eq:Moebius}.
\STATE Let $n_1$ be the number of these M\"obius transformations satisfying \eqref{new-cond}.
\STATE Find the gcd $R_0(\xi)$ of $g^\pm$ and the numerators of the coefficients of the polynomial $P_0(t)$ obtained by substituting $c = 1$, $d = 0$, \eqref{eq:ab} into $\Res_s(F, G^\pm)$.
\STATE Find the real roots of $R_0(\xi)$ for which \eqref{eq:kappatauts} is well defined, each defining a M\"obius transformation by substituting \eqref{eq:ab}, $c = 1$ and $d = 0$ in \eqref{eq:Moebius}.
\STATE Let $n_0$ be the number of these M\"obius transformations satisfying \eqref{new-cond}.
\STATE Return ``{\tt The curve has $n_0 + n_1$ symmetries with $\det(\mQ) = \pm 1$}''.
\end{algorithmic}
\caption*{{\bf Algorithm} {\tt Symm$^\pm$}}
\end{algorithm}

\subsection{The complete algorithm} \label{sec-alg}
\noindent Let $\bfx: \RR \dashrightarrow \CCC$ as in \eqref{eq:parametrizations} be a parametric curve of degree $m$. Distinguishing the cases $d = 0, 1$, each tentative M\"obius transformation can be written as \begin{equation*}
\varphi_\xi(t) = \frac{a_d(\xi)t + b_d(\xi)}{c_d(\xi)t + d},
\end{equation*}
with $\xi$ a root of $R_d$ and $a_d, b_d, c_d$ as in \eqref{eq:abcd=1}, \eqref{eq:ab}. Condition \eqref{new-cond} can be checked as follows. Squaring and clearing denominators yields an equivalent polynomial condition
\begin{equation}\label{eq:W}
W_\xi(t) = w_n(\xi)t^n + w_{n-1}(\xi)t^{n-1} + \cdots + w_0(\xi) \equiv 0
\end{equation}
of degree $n\leq 24m - 4$. By Theorem \ref{funda-sym}, a root $\xi$ of $R_d$ corresponds to a symmetry of $\CCC$ precisely when $W_\xi(t)$ vanishes identically. In other words, every root $\xi$ of
\begin{equation}\label{gcd}
\gcd (R_d, w_0,\ldots,w_n)
\end{equation}
defines a M\"obius transformation $\varphi_\xi$ corresponding to a symmetry $f_\xi := \bfx \circ \varphi_\xi \circ \bfx^{-1}$ as in Theorem~\ref{th-fund}. We thus arrive at Algorithm {\tt Symm$^\pm$} for determining the number of symmetries of the curve $\CCC$.

\section{Determining the symmetries} \label{sec-find}
\noindent Algorithm {\tt Symm$^\pm$} detects whether the parametric curve $\bfx$ from \eqref{eq:parametrizations} has nontrivial symmetries. In the affirmative case we would like to determine these symmetries. By Theorem~\ref{th-fund}, every such symmetry corresponds to a M\"obius transformation $\varphi = (at + b)/(ct + d)$, which corresponds to a M\"obius-like factor $F$ of $G$ computed by Algorithm {\tt Symm$^\pm$}. In this section we shall see how the symmetry $f(\bfx) = \mQ\bfx+\bfb$ can be computed from $\varphi$.

The commutative diagram in Theorem \ref{th-fund} describes the identity
\begin{equation} \label{fund-equality}
\mQ \bfx(t)+\bfb=\bfx\big(\varphi(t)\big).
\end{equation}
Let us distinguish the cases $d\neq 0$ and $d = 0$. In the latter case, \eqref{fund-equality} becomes
\[ \mQ\bfx(t) + \bfb = \bfx\big(\varphi(t)\big) = \bfx\big(\ta / t + \tb \big),\qquad \ta := b/c,\qquad \tb := a/c.\]
Applying the change of variables $t\longmapsto 1/t$ and writing $\tbfx(t) := \bfx(1/t)$, we obtain
\begin{equation} \label{spec}
\mQ\bfx(t) + \bfb =  \tbfx\big(\ta t + \tb\big).
\end{equation}
Without loss of generality, we assume that $\bfx(t)$ (respectively $\tbfx(t)$), and therefore any of its derivatives, is well defined at $t=\tb$ (respectively $t=0$), and that $\bfx'(0),\bfx''(0)$ are well defined, nonzero, and not parallel. The latter statement is equivalent to requiring that the curvature $\kappa_\bfx (t)$ at $t=0$ is well defined and distinct from $0$. This can always be achieved by applying a change of parameter of the type $t\longmapsto t+\alpha$. Observe that $\varphi(t)$ can be determined before applying this change, because afterwards the new M\"obius transformation is just $\varphi(t+\alpha)$.

Evaluating \eqref{spec} at $t = 0$ yields
\begin{equation}\label{eq:evaldzero}
\mQ\bfx(0) + \bfb = \tbfx(\tb),
\end{equation}
while differentiating once and twice and evaluating at $t=0$ yields
\begin{equation}\label{eq:diff12dzero}
\mQ\bfx'(0) = \ta \cdot \tbfx'(\tb),\qquad \mQ\bfx''(0) = \ta^2\cdot \tbfx''(\tb).
\end{equation}
Using \eqref{eq:cross} and that $\mQ$ is orthogonal, taking the cross product in \eqref{eq:diff12dzero} yields
\begin{equation}\label{eq:diff1xdiff2dzero}
\mQ\big( \bfx'(0)\times \bfx''(0) \big) = \det(\mQ)\cdot \ta^3\cdot \tbfx'(\tb) \times \tbfx''(\tb).
\end{equation}
Multiplying $\mQ$ by the matrix $\mB := [\bfx'(0),\ \bfx''(0),\ \bfx'(0) \times \bfx''(0)]$ therefore gives
\[ \mC := \big[\ta\cdot \tbfx'(\tb),\ \ta^2\cdot \tbfx''(\tb) ,\ \det(\mQ)\cdot \ta^3 \cdot \tbfx'(\tb) \times \tbfx''(\tb)\big] \]
and $\mQ = \mC\mB^{-1}$. One sets $\det(\mQ) = 1$ to find the orientation-preserving symmetries, and $\det(\mQ) = -1$ to find the orientation-reversing symmetries. One finds $\bfb$ from \eqref{eq:evaldzero}.

Next we address the case $d\neq 0$. After dividing the coefficients of $\varphi$ by $d$, we may assume $d=1$. As before, we assume that $\bfx(t)$ is well defined at $t=0$, and we again assume that the curvature $\kappa_\bfx (0)$ is well defined and nonzero. Differentiating \eqref{fund-equality} once and twice,
\begin{align}
\mQ\bfx'(t)  & = \bfx'\big(\varphi(t)\big)\cdot\varphi'(t) = \bfx'\left(\frac{a t + b}{c t + 1}\right)\frac{\Delta}{(c t + 1)^2}, \label{eq:first}\\
\mQ\bfx''(t) & = \bfx''\big(\varphi(t)\big)\big(\varphi'(t)\big)^2 + \bfx'\big(\varphi(t)\big)\varphi''(t) \label{eq:second}\\
           & =\displaystyle{\bfx''\left(\frac{a t + b }{c t + 1 }\right) \frac{\Delta^2}{(c t + 1 )^4} - 2\bfx'\left(\frac{a t + b}{c t + 1}\right) \frac{c \Delta}{(c t + 1)^3}}. \notag
\end{align}
Evaluating \eqref{eq:first} and \eqref{eq:second} at $t = 0$ yields
\begin{equation}\label{eq:diff12}
\mQ \bfx'(0)  = \bfx' (b)\cdot \Delta,\qquad
\mQ \bfx''(0) = \bfx''(b)\cdot \Delta^2 - 2 \bfx'(b)\cdot c\Delta.
\end{equation}
Using \eqref{eq:cross} and that $\mQ$ is orthogonal, taking the cross product in \eqref{eq:diff12} yields
\begin{equation}\label{eq:diff1xdiff2}
\mQ\big(\bfx'(0) \times \bfx''(0) \big)  = \det(\mQ)\cdot \Delta^3 \cdot \bfx'(b) \times \bfx''(b) .
\end{equation}
Since $\varphi$ is known, the matrix $\mQ$ can again be determined from its action on $\bfx'(0), \bfx''(0)$, and $\bfx'(0) \times \bfx''(0)$, which is given by Equations \eqref{eq:diff12} and \eqref{eq:diff1xdiff2}. One finds $\bfb$ by evaluating \eqref{fund-equality} at $t = 0$.

Once $\mQ$ and $\bfb$ are found, one can compute the set of fixed points of $f(\bfx) = \mQ\bfx + \bfb$ to determine the elements of the symmetry, i.e., the symmetry center, axis, or plane. 

\begin{example} \label{ex2}
Let $\CCC\subset \RR^3$ be the crunode space curve parametrized by
\[\bfx : t\longmapsto \left(\frac{t}{t^4+1},\frac{t^2}{t^4+1},\frac{t^3}{t^4+1}\right). \]
Applying Algorithm {\tt Symm$^+$} we get $G^+(t,s)=(t-s)(t+s)$. The first factor corresponds to the identity map $\varphi_1(t) = t$ and the trivial symmetry $f_1(\bfx) = \bfx$. The second factor corresponds to the M\"obius transformation $\varphi_2(t)=-t$. Clearly $\varphi_2$ satisfies Condition \eqref{new-cond}, so that Theorem \ref{funda-sym} implies that $\CCC$ has a nontrivial, direct symmetry $f_2(\bfx) = \mQ_2\bfx + \bfb_2$. With $a=-1$, $b=0$, $c=0$, $d=1$, and using that $\det(\mQ) = 1$,
\[ \mB = \left[\begin{array}{ccc} 1 & 0 & 0\\ 0 & 2 & 0\\ 0 & 0 & 2\end{array}\right],\
   \mC = \left[\begin{array}{ccr} -1 & 0 & 0\\ 0 & 2 & 0\\ 0 & 0 & -2\end{array}\right],\
   \mQ_2 := \mC \mB^{-1} = \left[\begin{array}{ccr} -1 & 0 & 0\\ 0 & 1 & 0\\ 0 & 0 & -1\end{array}\right]. \]
Evaluating \eqref{fund-equality} at $t = 0$ gives $\bfb_2 = (\mI - \mQ_2)\bfx(0)=0$, so that $\CCC$ is invariant under $f_2(\bfx) = \mQ_2\bfx$, which is a half-turn about the $y$-axis. Since there are no other factors in $G^+$, there are no direct symmetries corresponding to a M\"obius transformation with $d=0$.

As for the opposite symmetries, applying Algorithm {\tt Symm$^-$} yields $G^-(t,s) = (st-1)(st+1)$, whose factors correspond to the M\"obius transformations $\varphi_3(t)=1/t$ and $\varphi_4(t)=-1/t$. A direct computation shows that $\varphi_3$ and $\varphi_4$ satisfy Condition \eqref{new-cond}, and that they correspond to symmetries $f_3(\bfx) = \mQ_3 \bfx$ and $f_4(\bfx) = \mQ_4 \bfx$, with
\[ \mQ_3 = \left[\begin{array}{ccr} 0 & 0 &  1\\ 0 & 1 & 0\\  1 & 0 & 0\end{array}\right],\qquad
   \mQ_4 = \left[\begin{array}{ccr} 0 & 0 & -1\\ 0 & 1 & 0\\ -1 & 0 & 0\end{array}\right]. \]
The sets of fixed points of these isometries are the planes $\Pi_3 : z - x = 0$ and $\Pi_4 : z + x = 0$, which intersect in the symmetry axis of the half-turn; see Figure~\ref{fig:crunode}.
\end{example}

\begin{figure}
\begin{center}
\includegraphics[scale=0.16]{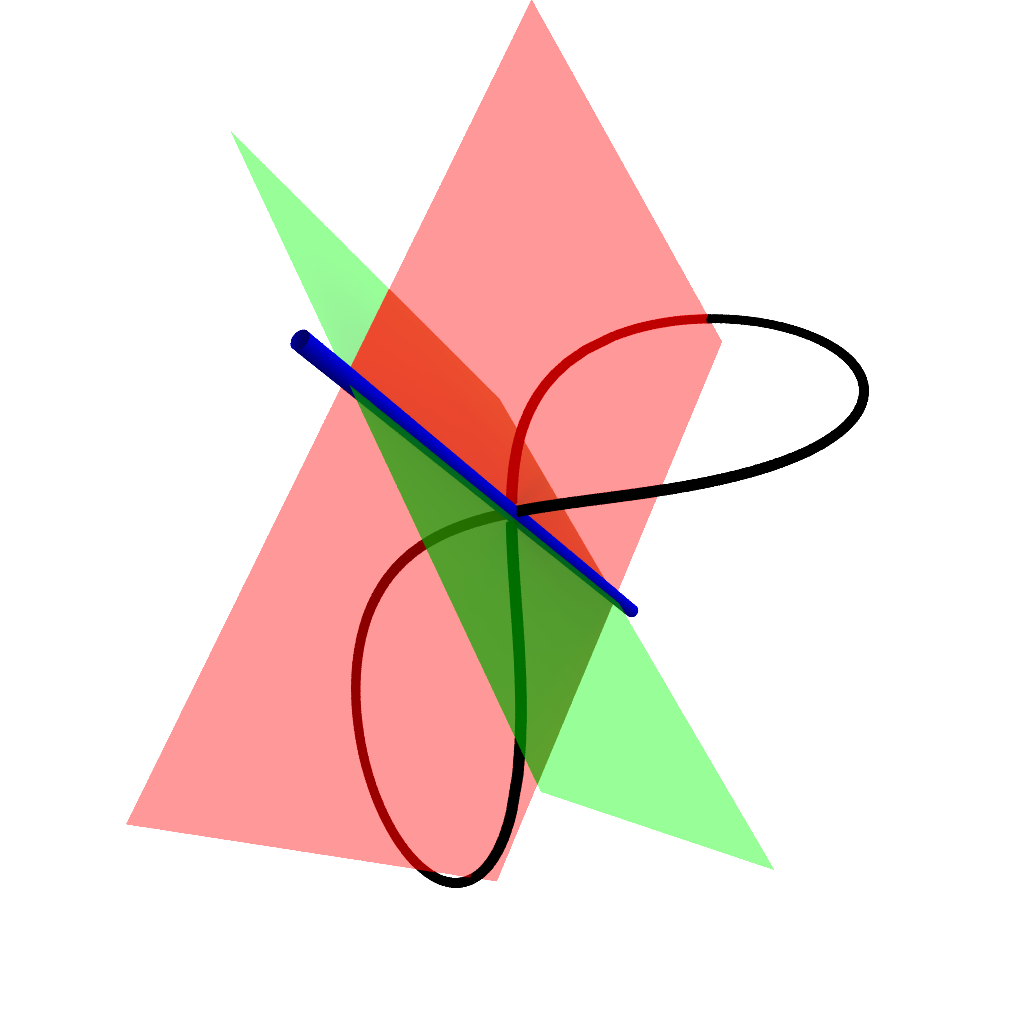}
\includegraphics[scale=0.16]{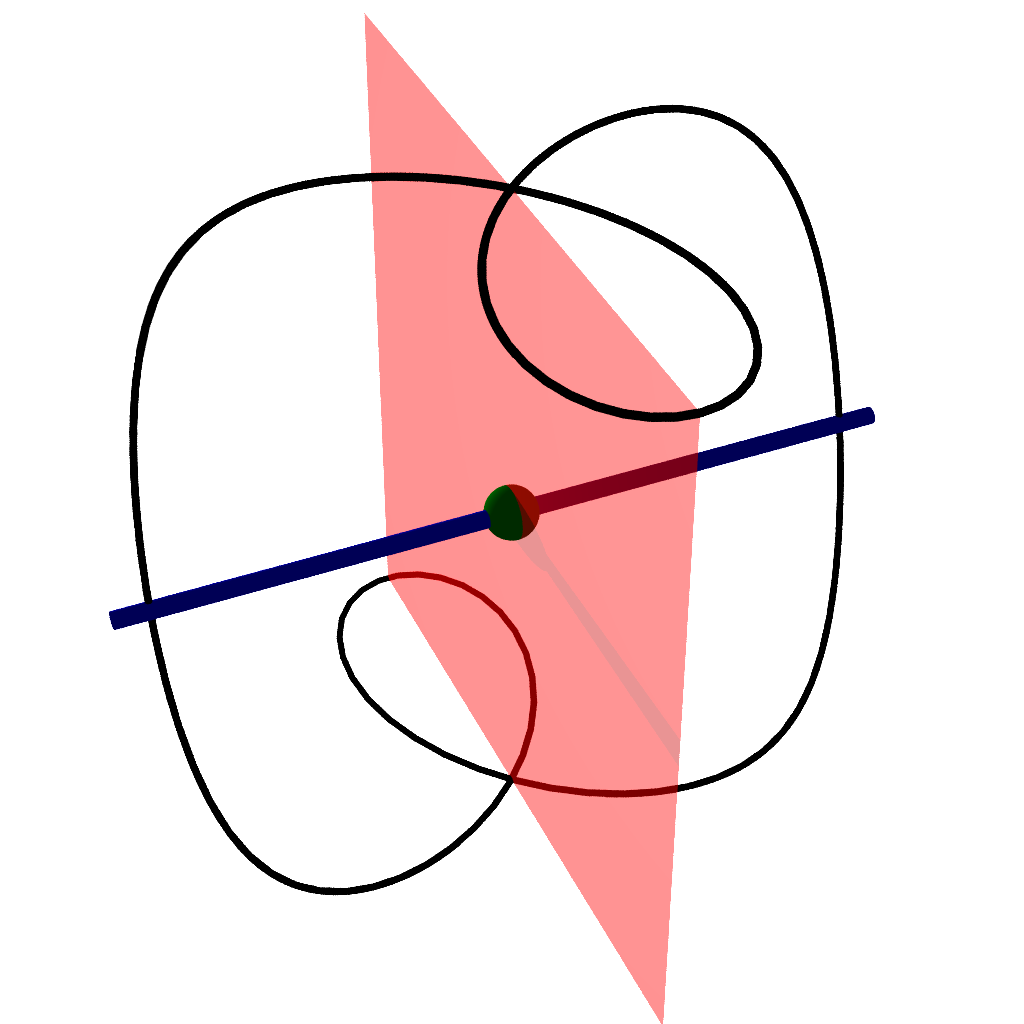}
\end{center}
\caption{Left: The crunode curve from Example \ref{ex2}, together with the fixed points of the half-turn and mirror symmetries. Right: The daisy of degree 8 from Example \ref{ex:daisy}, together with the fixed points of the central inversion, half-turn, and mirror symmetries.}\label{fig:crunode}
\end{figure}

\begin{example}\label{ex:daisy}
Consider the family of \emph{daisies} of increasing degree $m = 4j + 4$, which are given parametrically by
\begin{equation}\label{eq:daisies}
\bfx(t) = \left(
u\sum_{i=0}^j (-1)^i {2j\choose 2i} u^{2j - 2i} v^{2i},
v\sum_{i=0}^j (-1)^i {2j\choose 2i} u^{2j - 2i} v^{2i},
\frac{1 - t^{4j + 4}}{1 + t^{4j + 4}}
\right),
\end{equation}
with
\[ u = \frac{1-t^2}{1+t^2},\qquad v = \frac{2t}{1+t^2}, \qquad j = 0, 1, \ldots \]
Applying Algorithm {\tt Symm$^+$} for the case $j = 1$, we get $G^+(t,s) = (t-s)(st - 1)$. The first factor again corresponds to the trivial symmetry $f_1(\bfx) = \bfx$. The second factor corresponds to the M\"obius transformation $\varphi_2(t) = 1/t$. Clearly $\varphi_2$ satisfies Condition \eqref{new-cond}, so that Theorem \ref{funda-sym} implies that $\CCC$ has a nontrivial, direct symmetry $f_2(\bfx) = \mQ_2\bfx + \bfb_2$. With $a=0$, $b=1$, $c=1$, $d=0$, and using that $\det(\mQ) = 1$,
\[ \mB = \left[\begin{array}{ccc} 0 & -20 & 0\\ 2 & 0 & 0\\ 0 & 0 & 40\end{array}\right],\
   \mC = \left[\begin{array}{ccr} 0 & 20 & 0\\ 2 & 0 & 0\\ 0 & 0 & -40\end{array}\right],\
   \mQ_2 := \mC \mB^{-1} = \left[\begin{array}{ccr} -1 & 0 & 0\\ 0 & 1 & 0\\ 0 & 0 & -1\end{array}\right]. \]
Equation \eqref{eq:evaldzero} gives $\bfb_2 = \tbfx(\tb) - \mQ_2\bfx(0) = 0$, so that $\CCC$ is invariant under $f_2(\bfx) = \mQ_2\bfx$, which is a half-turn about the $y$-axis.

Similarly applying Algorithm {\tt Symm$^-$}, we get $G^-(t,s) = (s+t)(st+1)$, whose factors correspond to the M\"obius transformations $\varphi_3(t)=-t$ and $\varphi_4(t)=-1/t$. A direct computation shows that $\varphi_3$ and $\varphi_4$ satisfy Condition \eqref{new-cond}, and that they correspond to symmetries
\[ f_3(\bfx) = \begin{bmatrix}  1 & 0 & 0\\ 0 & -1 & 0\\ 0 & 0 &  1 \end{bmatrix}\bfx,\qquad 
   f_4(\bfx) = \begin{bmatrix} -1 & 0 & 0\\ 0 & -1 & 0\\ 0 & 0 & -1 \end{bmatrix}\bfx, \]
which are a reflection in the plane $\Pi_3: y = 0$ and a central inversion about the point $(0,0,0)$, respectively; see Figure~\ref{fig:crunode}.
\end{example}

\section{Performance} \label{exp-sec}

\subsection{Complexity} \label{sec:complexity}
\noindent 
Let us determine the \emph{arithmetic} complexity of Algorithm {\tt Symm$^\pm$}, i.e., the number of integer operations needed. In addition to using the standard Big O notation $\OOO$ for the space- and time-complexity analysis, we use the \emph{Soft O} notation $\tOOO$ to ignore any logarithmic factors in the time-complexity analysis. The \emph{bitsize} $\tau$ of an integer $k$ is defined as $\tau = \lceil \log_2 k \rceil + 1$; the bitsize of a parametrization~$\bfx$ (taken with integer coefficients) is the maximum bitsize of the coefficients of the  numerators and denominators of the components. The following theorem presents the arithmetic complexity of Algorithm {\tt Symm$^\pm$} when applied to parametric curves of varying degree $m$ and of fixed bitsize.

\begin{theorem}
For a parametric curve $\bfx$ as in \eqref{eq:parametrizations} with degree $m$, Algorithm {\tt Symm$^\pm$} finishes in $\tOOO(m^5)$ integer operations.
\end{theorem}

\begin{proof}
\emph{Step 1}. Using the Sch\"onhage-Strassen algorithm, two polynomials of degree $m$ with integer coefficients can be multiplied in $\tOOO(m)$ operations \cite[Table 8.7]{VonZurGathen.Gerhard}. Therefore the computation of $\kappa^2$ and $\tau$ can be carried out in $\tOOO(m)$ operations as well, resulting in rational functions whose numerators and denominators have degree $\OOO(m)$. As a consequence, $K$ and $T = T^\pm$ can also be computed in $\tOOO(m)$ operations, and have degrees $\OOO(m)$ in $t$ and $s$. The bivariate gcd $G = G^\pm$ can be computed in $\tOOO(m^5)$ operations using the `half-gcd algorithm' \cite{Reischert}, and has degree $\OOO(m)$ in both variables. Step 1 therefore takes at most $\tOOO(m^5)$ operations.

\emph{Step 2}. Since $F(t,s) = (ct + d)s - (at + b)$, the resultant $\Res_s(F,G)$ is the polynomial in $t$ obtained by replacing $s$ by $(at + b)/(ct + d)$ and clearing denominators. Writing $G(t, s) = \sum_{k=0}^{m_0} G_k(t) s^k$ as a sum of $m_0 + 1 = \OOO(m)$ terms, with each
$G_k(t)$ a polynomial of degree $\OOO(m)$, gives
\begin{align}\label{eq:resFG}
\Res_s(F, G) &= \sum_{k=0}^{m_0} G_k(t) ( a t + b )^k ( c t + 1 )^{m_0 - k},
\end{align}
which is a polynomial of degree $\OOO(m)$ in $a,b,c$, and $t$.

\emph{Step 3}. For any integer $t_0$, the two polynomials $G(t_0, s)$, $G_s(t_0, s)$ have degree $\OOO(m)$, so that their (univariate) gcd can be computed in $\tOOO(m)$ operations \cite[Corollary 11.6]{VonZurGathen.Gerhard}. Since we need to consider at most $\OOO(m^2)$ values of $t_0$, Step 3 takes $\tOOO(m^3)$ operations.

\emph{Step 4}. For any integer $t_0$ and unknown $\xi$, evaluating \eqref{eq:difffunc} at $(t_0, \xi)$ takes $\OOO(m^2)$ operations, yielding rational functions $s_0'(\xi)$ and $s_0''(\xi)$ whose numerator and denominator have degree $\OOO(m)$. Substituting these rational functions into \eqref{eq:abcd=1} takes $\tOOO(m)$ operations and yields rational functions
\begin{equation}\label{eq:aibici}
   a(\xi) = \frac{a_1(\xi)}{a_2(\xi)},\qquad
   b(\xi) = \frac{b_1(\xi)}{b_2(\xi)},\qquad
   c(\xi) = \frac{c_1(\xi)}{c_2(\xi)},
\end{equation}
whose numerator and denominator have degree $\OOO(m)$. Substituting these rational functions into \eqref{eq:resFG} followed by binomial expansion, i.e., computing
\begin{equation}\label{eq:binexp}
\frac{\displaystyle \sum_{n = 0}^k {k\choose n} a_1^n b_2^{m_0 - k + n} b_1^{k-n} a_2^{m_0 - n} t^n}{a_2^{m_0} b_2^{m_0}},\qquad
\frac{\displaystyle \sum_{n = 0}^{m_0 - k} {m_0 - k\choose n} c_1^n c_2^{m_0-n} t^n}{c_2^{m_0}},
\end{equation}
involves raising polynomials of degree $\OOO(m)$ to the power $\OOO(m)$, which can be computed in $\tOOO(m^2)$ operations using repeated squaring, i.e., $\OOO(\log m)$ multiplications of polynomials of degree $\OOO(m^2)$. All powers $a_i^l, b_i^l, c_i^l$, with $i = 1,2$ and $l = 0, \ldots, m_0$, in the above expression can therefore be computed in $\tOOO(m^3)$ operations, resulting in polynomials of degree $\OOO(m^2)$ in $\xi$. All remaining products can be computed in $\tOOO(m^3)$ operations, resulting in polynomials of degree $\OOO(m^2)$.

Now the rational functions in \eqref{eq:binexp} are determined, the product of their numerators can be carried out in $\tOOO(m)$ ring operations, the ring now being the polynomials in $\xi$. Since these polynomials have degree $\OOO(m^2)$, the product of $G_k(t)$ and the rational functions in \eqref{eq:binexp} takes $\tOOO(m^2)$ integer operations, and yields the terms in the sum \eqref{eq:resFG}. After factoring out the common denominator $(a_2 b_2 c_2)^{m_0}$, this sum involves $\OOO(m)$ polynomials of degree $\OOO(m)$ in $t$, which requires $\OOO(m^2)$ additions of polynomials of degree $\OOO(m^2)$ in $\xi$. This involves $\OOO(m^4)$ integer operations and yields a polynomial of degree $\OOO(m)$ in $t$, whose coefficients $P_i(\xi)$ are polynomials of degree $\OOO(m^2)$. The gcd of $g(\xi)$ with the $P_i(\xi)$ can be computed in $\tOOO(m^3)$ operations, resulting in a polynomial $R_1(\xi)$ of degree $\OOO(m)$, since $g(\xi)$ has degree $\OOO(m)$. Step~4 therefore takes $\OOO(m^4)$ operations.

\emph{Step 5}. One determines whether $R_1(\xi)$ has real roots using root isolation, which takes $\tOOO(m)$ operations using Pan's algorithm for root isolation \cite{Pan, Mehlhorn}.

\emph{Step 6}. Writing $\bfx = (x, y, z) = (x_1/x_2, y_1/y_2, z_1/z_2)$, we find that
\[ \|\bfx'\|^2 = \frac{(x_2x_1' - x_1x_2')^2 y_2^4 z_2^4 + (y_2y_1' - y_1y_2')^2 x_2^4 z_2^4 + (z_2z_1' - z_1z_2')^2 x_2^4 y_2^4}{x_2^4 y_2^4 z_2^4} \]
can be computed in $\OOO(m)$ operations. From Step 3 we already know the expansions of the powers $(at + b)^l, (ct + d)^l$ and their products, thus determining $x\circ \varphi$. Taking the derivative of $x\circ \varphi$ and then squaring involves multiplying and adding polynomials of degree $\OOO(m)$ in $t$ and $\OOO(m^2)$ in $\xi$, which requires $\tOOO(m^2)$ operations. Similarly we determine $[(y \circ \varphi)']^2$ and $[(z \circ \varphi)']^2$ in $\tOOO(m^2)$ operations. The resulting rational functions have numerator and denominator of degree $\OOO(m)$ in $t$ and $\OOO(m^2)$ in $\xi$, and can be added in $\tOOO(m^2)$ operations. Clearing denominators again takes $\tOOO(m^2)$ operations and results in the polynomial $W_\xi(t)$ from \eqref{eq:W} of degree $\OOO(m)$ in $t$ and of degree $\OOO(m^2)$ in $\xi$. To compute \eqref{gcd}, we need to compute $\OOO(m)$ times the univariate gcd of polynomials of degree $\OOO(m)$ and degree $\OOO(m^2)$, which requires $\tOOO(m^3)$ operations.
Step~6 therefore requires $\tOOO(m^3)$ operations.

\emph{Steps 7--10}. These steps have the same complexity as Steps 4--6.
\end{proof}

Note that resorting to probabilistic algorithms, the bivariate gcd $G$ in Step~1 can be computed in $\tOOO(m^2)$ operations using the `small primes modular gcd algorithm' and fast polynomial arithmetic \cite[Corollary 11.9.(i)]{VonZurGathen.Gerhard}. Thus a probabilistic version of Algorithm {\tt Symm$^\pm$} uses $\OOO(m^4)$ operations.

\begin{table}
\begin{tabular*}{\columnwidth}{@{ }@{\extracolsep{\stretch{1}}}*{1}{lcrr}@{ }}
\toprule
curve & degree & $t_{\mbox{old}}$ & $t_{\mbox{new}}$ \\
\midrule
twisted cubic &  3 &  0.26 & 0.15 \\
cusp          &  4 &  0.52 & 0.16 \\
half-turn 1   &  4 &  2.22 & 0.14 \\
crunode       &  4 & 39.60 & 0.42 \\
inversion 1   &  7 &  6.80 & 0.19 \\
space rose    &  8 & 57.60 & 0.25 \\
inversion 2   & 11 & 75.60 & 0.22 \\
\bottomrule
\end{tabular*}
\caption{Average CPU time (seconds) of the algorithm in \cite{AHM13-2} ($t_\text{old}$) and Algorithm {\tt Symm$^\pm$} ($t_{\text{new}}$) for parametric curves given in \cite{AHM13-2}.}\label{tab:spaceinvolutions}
\end{table}

\subsection{Experimentation} \label{exp}
\noindent Algorithm {\tt Symm$^\pm$} was implemented in the computer algebra system {\tt Sage} \cite{sage}, using {\tt Singular} \cite{singular} as a back-end, and was tested on a Dell XPS 15 laptop, with 2.4 GHz i5-2430M processor and 6 GB RAM. Additional technical details are provided in the {\tt Sage} worksheet, which can be downloaded from the third author's website \cite{WebsiteGeorg} and can be tried out online by visiting {\tt SageMathCloud}~\cite{SageMathCloud}.

\begin{table}
\begin{tabular*}{\columnwidth}{@{ }@{\extracolsep{\stretch{1}}}*{1}{cccccc}@{ }}
\toprule
& \hspace{-0.5em}\includegraphics[scale = 0.1]{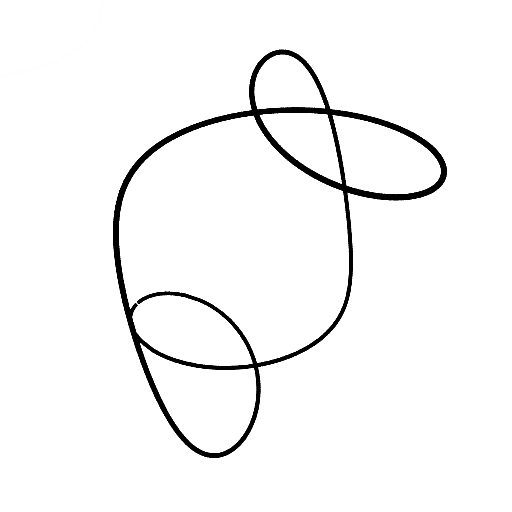}\hspace{-0.5em} &
  \hspace{-0.5em}\includegraphics[scale = 0.1]{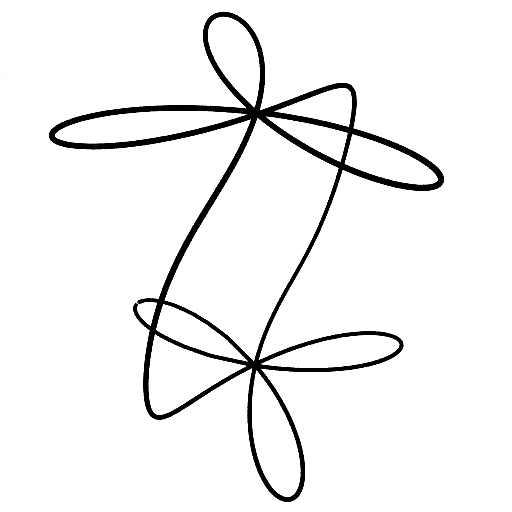}\hspace{-0.5em} &
  \hspace{-0.5em}\includegraphics[scale = 0.1]{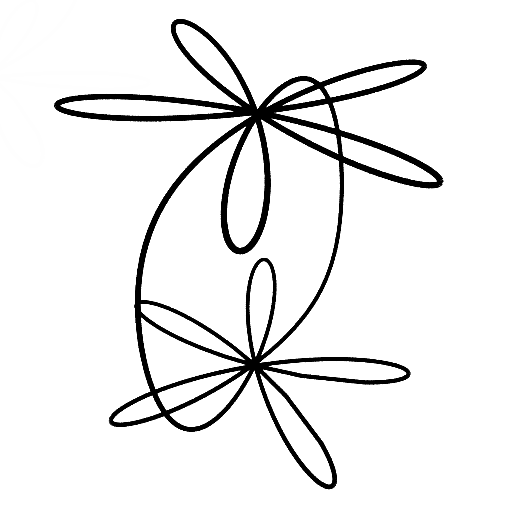}\hspace{-0.5em} &
  \hspace{-0.5em}\includegraphics[scale = 0.1]{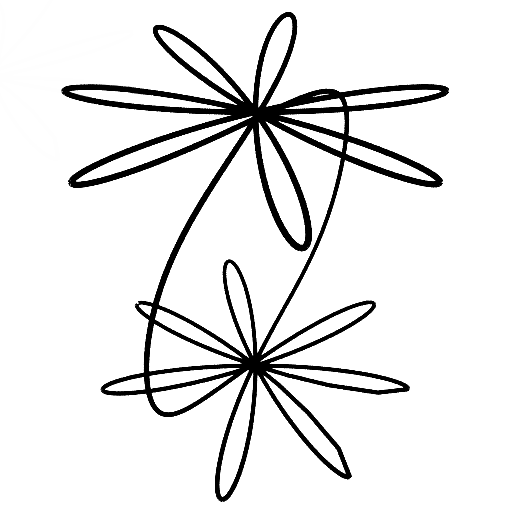}\hspace{-0.5em} &
  \hspace{-0.5em}\includegraphics[scale = 0.1]{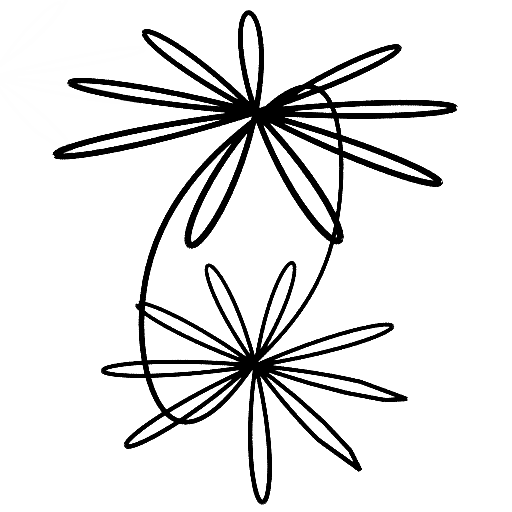}\hspace{-0.5em} \\
degree         &    8 &   12 &   16 &   20 &   24 \\
$t_\text{new}$ & 0.66 & 0.92 & 1.47 & 2.30 & 4.38 \\
\midrule
& \hspace{-0.5em}\includegraphics[scale = 0.1]{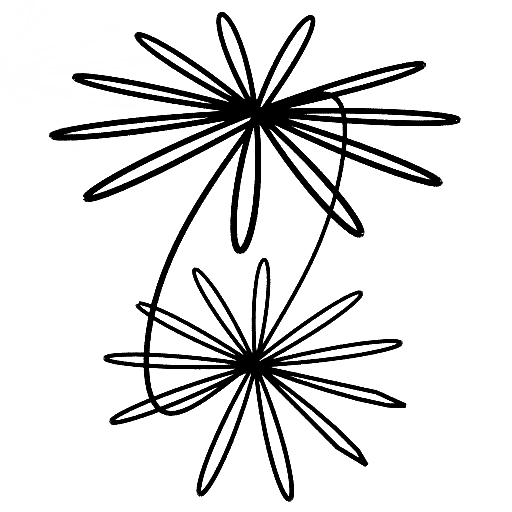}\hspace{-0.5em} &
  \hspace{-0.5em}\includegraphics[scale = 0.1]{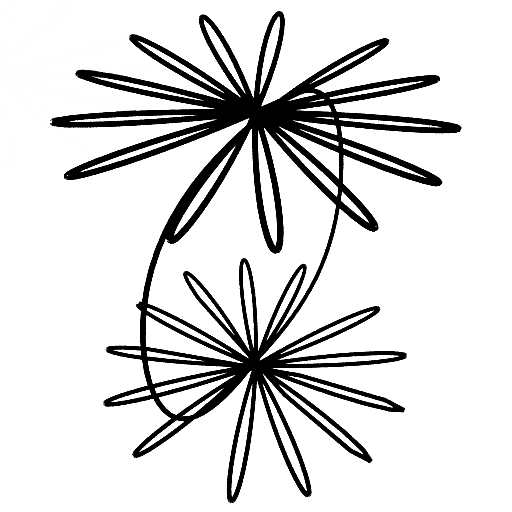}\hspace{-0.5em} &
  \hspace{-0.5em}\includegraphics[scale = 0.1]{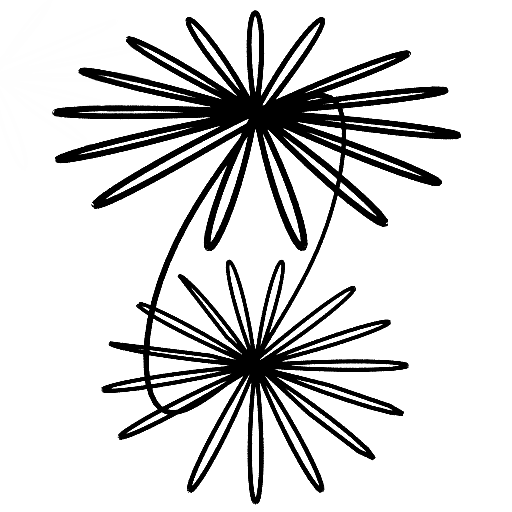}\hspace{-0.5em} &
  \hspace{-0.5em}\includegraphics[scale = 0.1]{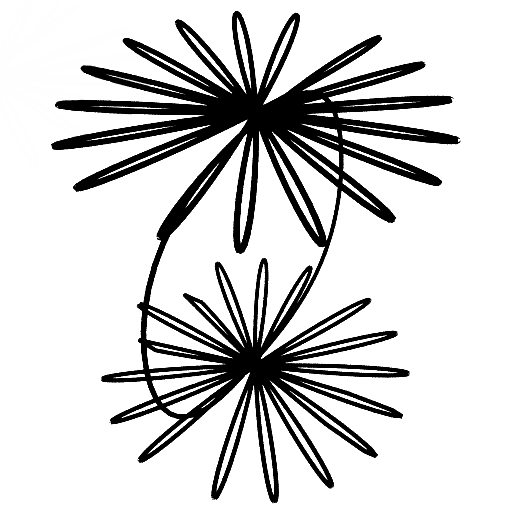}\hspace{-0.5em} &
  \hspace{-0.5em}\includegraphics[scale = 0.1]{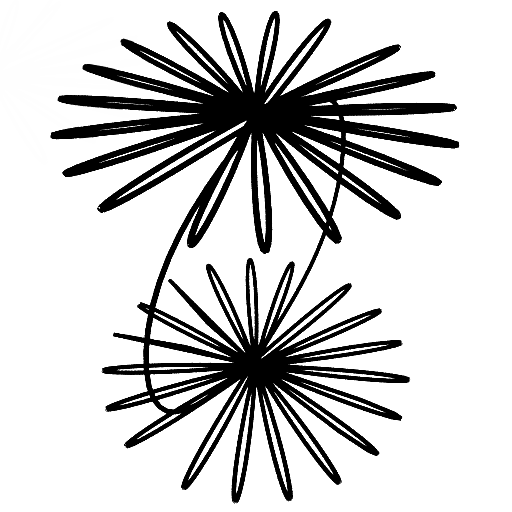}\hspace{-0.5em} \\
degree         &   28 &   32 &   36 &    40 &    44 \\
$t_\text{new}$ & 5.33 & 6.53 & 8.77 & 15.88 & 18.11 \\
\bottomrule
\end{tabular*}
\caption{Average CPU time $t_\text{new}$ (seconds) of Algorithm {\tt Symm$^\pm$} for daisies of various degrees.}\label{tab:spacerose}
\end{table}

We present tables with timings corresponding to different groups of examples. Table \ref{tab:spaceinvolutions} corresponds to the set of examples in \cite{AHM13-2}, making it possible to compare the timing $t_\mathrm{new}$ of Algorithm {\tt Symm$^\pm$} to the timing $t_\mathrm{old}$ of the algorithm in \cite{AHM13-2}. It is clear from the table that the algorithm introduced in this paper is considerably faster for each curve.

To test Algorithm {\tt Symm$^\pm$} for symmetric curves with higher degree, Table~\ref{tab:spacerose} lists the timings for a family of daisies of increasing degree $m = 4j + 4$, parametrically given by \eqref{eq:daisies}. The algorithm quickly finds the symmetries of these symmetric curves, also for high degree.

\begin{table}[t!]
\begin{tabular*}{\columnwidth}{m{5em}@{\extracolsep{\stretch{1}}}*{1}{rrrrrrrrrr}}
\toprule
$t_\text{new}$ &  $\tau = 4$ &  $\tau = 8$ & $\tau = 16$ & $\tau = 32$ & $\tau = 64$ & $\tau = 128$ & $\tau = 256$ \\ \midrule 
$m=4$  &          0.61 &          0.62 &          0.66 &          0.73 &          0.83 &          1.14 &          1.88 \\ 
$m=6$  &          1.65 &          1.76 &          1.72 &          1.89 &          2.13 &          2.80 &          4.55 \\ 
$m=8$  &          3.50 &          3.54 &          3.55 &          3.84 &          4.27 &          5.36 &          8.59 \\ 
$m=10$ &          7.53 &          7.47 &          7.30 &          7.98 &          8.42 &          9.76 &         15.35 \\ 
$m=12$ &         14.46 &         14.35 &         14.30 &         14.84 &         15.98 &         18.35 &         25.87 \\ 
$m=14$ &         22.31 &         23.24 &         22.39 &         22.71 &         24.93 &         27.35 &         38.36 \\ 
$m=16$ &         34.86 &         35.60 &         35.38 &         35.27 &         38.14 &         41.91 &         55.74 \\ 
$m=18$ &         53.03 &         52.78 &         52.78 &         51.16 &         54.49 &         60.44 &         78.27 \\ 
\bottomrule
\end{tabular*}
\caption{CPU times $t_{\text{new}}$ (seconds) for random dense rational parametrizations of various degrees $m$ and coefficients with bitsize bounded by $\tau$.}\label{tab:sizedegree}
\end{table}

\begin{table}[t!]
\begin{tabular*}{\columnwidth}{m{5em}@{\extracolsep{\stretch{1}}}*{1}{rrrrrrrrrr}}
\toprule
$t_\text{new}$ &  $\tau = 4$ &  $\tau = 8$ & $\tau = 16$ & $\tau = 32$ & $\tau = 64$ & $\tau = 128$ & $\tau = 256$ \\ \midrule 
$m=4$  &      0.85 &      0.89 &      0.91 &      0.97 &      1.12 &      1.51 &      2.39 \\ 
$m=6$  &      1.89 &      2.02 &      1.99 &      2.21 &      2.57 &      3.36 &      5.40 \\ 
$m=8$  &      4.08 &      4.24 &      4.52 &      5.16 &      5.45 &      7.42 &     10.41 \\ 
$m=10$ &      8.29 &      8.80 &      8.88 &      9.30 &     10.74 &     12.64 &     19.87 \\ 
$m=12$ &     17.47 &     18.20 &     17.96 &     17.12 &     18.49 &     25.19 &     34.11 \\ 
$m=14$ &     28.54 &     28.72 &     29.55 &     28.54 &     31.87 &     34.71 &     44.19 \\ 
$m=16$ &     41.20 &     41.53 &     42.02 &     43.07 &     45.55 &     51.36 &     65.58 \\ 
$m=18$ &     58.42 &     58.91 &     59.54 &     61.08 &     64.31 &     71.89 &     94.33 \\ 
\bottomrule
\end{tabular*}
\caption{CPU times $t_{\text{new}}$ (seconds) for random dense rational parametrizations with a central inversion of various degrees $m$ and coefficients with bitsize bounded by $\tau$.}\label{tab:sizedegreeinv}
\end{table}

\begin{figure}[t!]
\begin{center}
\includegraphics[scale=0.7]{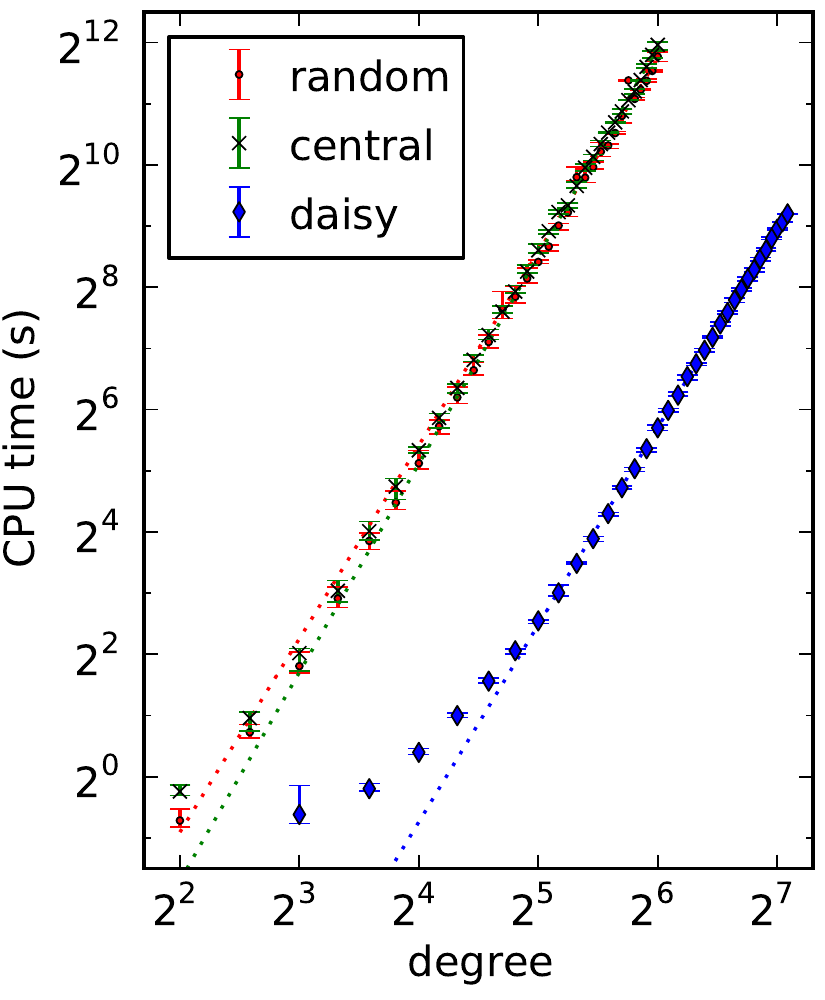}\quad
\includegraphics[scale=0.7]{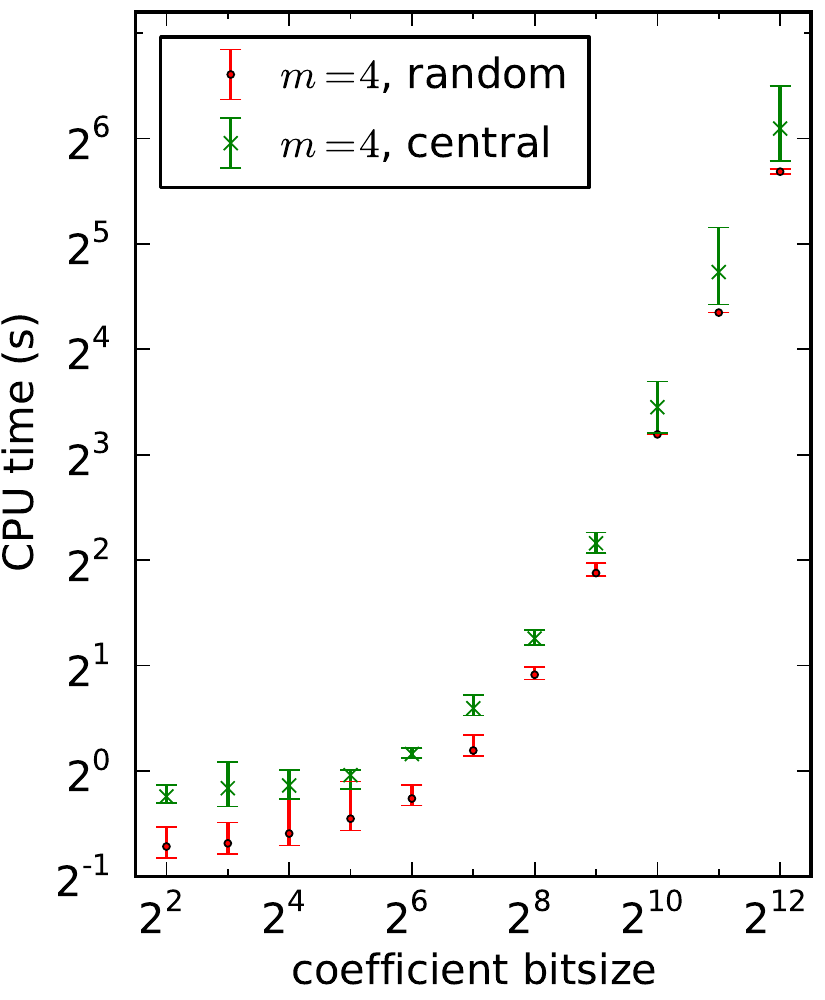}
\end{center}
\caption{Left: The average CPU time $t_{\text{new}}$ (seconds) versus the degree $m$ for daisies and random dense rational parametrizations with bitsize $\tau = 4$ and with only the trivial symmetry (random) and with an additional central inversion (central), fitted by the power laws \eqref{eq:powerlaw1}--\eqref{eq:powerlaw2} (dotted). Right: The average CPU time $t_{\text{new}}$ (seconds) versus the coefficient bitsizes for degree $m=4$. The symbols indicate the average CPU times, and the error bars show the interval of CPU times.}\label{fig:timings}
\end{figure}

Table \ref{tab:sizedegree} lists average timings for random dense rational parametrizations with various degrees $m$ and coefficients with bitsizes at most $\tau$. To study the effect of an additional nontrivial symmetry, we consider random parametrizations $\bfx = (x_1, x_2, x_3)$ with antisymmetric numerators and symmetric denominators of the same degree $m$ and with bitsize at most $\tau$, i.e., of the form
\[
x_i(t) = \frac
{c_{i,0} + c_{i, 1} t + \cdots - c_{i, 1} t^{m-1} - c_{i, 0} t^m}
{d_{i,0} + d_{i, 1} t + \cdots + d_{i, 1} t^{m-1} + d_{i, 0} t^m},
\qquad i = 1,2,3,
\]
with $\lceil \log_2 |c_{i,j}| \rceil, \lceil \log_2 |d_{i,j}| \rceil \leq \tau - 1$. Since $\bfx(1/t) = - \bfx(t)$, such parametric curves have a central inversion about $\bfx(1) = 0$. Table \ref{tab:sizedegreeinv} lists average timings for these curves with various degrees $m$ and bitsizes at most $\tau$.

For very large coefficient bitsizes ($> 256$, i.e., coefficients with more than 77 digits) and high degrees ($> 20$) the machine runs out of memory. We have therefore analyzed separately the regime with high degree and the regime with large coefficient bitsize. 

Figure~\ref{fig:timings} presents log-log plots of the CPU times against the degree (left) and against the coefficient bitsizes (right). The (eventually) linear nature of these data suggests the existence of an underlying power law. Least squares approximation yields that, as a function of the degree $m$, the average CPU time $t_\text{new}$ satisfies
\begin{equation}\label{eq:powerlaw1}
t_\text{new} \sim \alpha m^\beta, \qquad \alpha \approx 6.7\cdot 10^{-3},\qquad \beta\approx 3.2
\end{equation}
in case of random dense rational parametrizations with coefficient bitsize at most $\tau = 4$,
\begin{equation}\label{eq:powerlaw1b}
t_\text{new} \sim \alpha m^\beta, \qquad \alpha \approx 4.7\cdot 10^{-3},\qquad \beta\approx 3.3,
\end{equation}
in case of random dense rational parametrizations with a central inversion and with coefficient bitsize at most $\tau = 4$, and
\begin{equation}\label{eq:powerlaw2}
t_\text{new} \sim \alpha m^\beta, \qquad \alpha \approx 7.9\cdot 10^{-5},\qquad \beta\approx 3.2
\end{equation}
for the daisies. Note that these timings are close to the $\tOOO(m^3)$ operations needed by Brown's modular gcd algorithm \cite{Brown}, which is used in the implementation in {\tt Sage} for bivariate gcd computations. The reason is that in the analyzed examples almost all time is spent computing the bivariate gcd in Step 1, which then typically has low degree, so that the remaining calculations take relatively little time.

\subsection{An observation on plane curves} \label{obs-planar}
\noindent If $\CCC$ is planar, then $\tau$ and $T^\pm$ are identically zero, so that $G^\pm = K$. Although Algorithm {\tt Symm$^\pm$} is still valid for such curves, we have observed a very poor performance in this case. The reason is that, for non-planar curves, the degree of $G^\pm$ is typically small compared to the degrees of $K$ and $T^\pm$. However, for plane curves the degree of $G^\pm$ is equal to the degree of $K$, and then the computation takes a very long time. Therefore, for plane curves, the algorithms in \cite{AHM13} and \cite{AHM13-2} are preferable.

\subsection{Comparison to previous method}
\noindent Table \ref{tab:spaceinvolutions} indicates a dramatic improvement of the CPU time of {\tt Symm$^\pm$} over the method described in \cite{AHM13-2}. In that paper the symmetry $f(\bfx) = Q\bfx + \bfb$ and M\"obius transformation $\varphi(t) = (at + b)/(ct + d)$ are first expressed polynomially in terms of some (yet unknown) algebraic number $\beta$. By far the most CPU time is spent after that, on substituting $a(\beta), b(\beta), c(\beta), Q(\beta)$ and $\mathbf{b}(\beta)$ into the relation
\begin{equation*}\label{eq:fundamentalrelation}
f\big(\bfx(t)\big) - \bfx\big(\varphi(t)\big) \equiv 0.
\end{equation*}
Since the degrees can get very high in this relation, this substitution can take a long time. Then the algebraic numbers $\beta$, and therefore the symmetry and M\"obius transformation, are found by requiring that this relation holds identically.

Furthermore, the method described in \cite{AHM13-2} requires that the parametrization $\bfx$ satisfies rather strict conditions. Quite often, a reparametrization is needed in order to achieve these conditions, which can result in destroying sparseness and increasing the coefficient size. This, in turn, has an impact on the time taken by the substitution step.

By contrast, in Algorithm {\tt Symm$^\pm$} we use additional information provided by the curvature and torsion of the curve to compute $G^\pm_{\bfx}=\gcd(K_{\bfx}, T^\pm_{\bfx})$, whose degree is generally low. The M\"obius transformations are then computed in just one step as factors of $G^\pm_{\bfx}$. As a consequence, no	 substitution step is needed. Moreover, unless $\bfx$ is not proper, no reparametrization is required.

\section{Conclusion} \label{sec-conclusions}
\noindent We have presented a new, deterministic, and efficient method for detecting whether a rational space curve is symmetric. The method combines ideas in \cite{A13, AHM13-2} with the use of the curvature and torsion as differential invariants of space curves. The complexity analysis and experiments show a good theoretical and practical performance, clearly beating the performance obtained in \cite{AHM13-2}. The algorithm also improves in scope on the algorithm of \cite{AHM13-2}, which can only be applied to find the symmetries for Pythagorean-hodograph curves and involutions of other curves. Finally Algorithm {\tt Symm$^\pm$} is simpler than the algorithm in \cite{AHM13-2}, which imposes certain conditions on the parametrization that often lead to a reparametrized, non-sparse curve whose coefficients have a large bitsize. By contrast, the algorithm in this paper has fewer requirements and is efficient even with high degrees.

Note that Algorithm {\tt Symm$^\pm$} is based on two conditions in Theorem \ref{funda-sym}, one involving the curvature and torsion of the curve and the other one involving the arc length. One might wonder whether these two conditions really are independent for the case of rational curves. We included both conditions because we did not succeed in proving that they are dependent, but neither did we find an example of a tentative M\"obius transformation not satisfying Condition \eqref{new-cond}. The relation between the conditions is therefore undetermined, and we pose the question here as an open problem. In any case, in the complexity analysis and experiments we observed that the cost of checking \eqref{new-cond} is small compared to the rest of the algorithm.

The implementation in {\tt Sage} can be improved in several ways. First, several of the methods named in the complexity section are not included in {\tt Sage}, which carries out the corresponding tasks by using other algorithms. Furthermore, almost all space curves are asymmetric, and these cases can be identified faster. In order to do this, one can first remove all factors $s-t$ from $K_\bfx$ and $T^\pm_{\bfx}$, and then check whether the remaining polynomials are coprime using modular arithmetic. In the affirmative case, the conclusion that the curve has no nontrivial symmetries could be obtained at very little computational cost.

As a final remark, as this paper sets forth a method for computing exact symmetries of parametric curves with rational coefficients, one could ask whether a similar development could yield a method for computing approximate symmetries of parametric curves with floating point coefficients. This is an open question that we would like to address in the future.

\section*{Acknowledgments}
\noindent We are very grateful for the detailed reports of the reviewers, which helped to improve the paper significantly, in particular Section \ref{exp-sec}. We also thank Rob Corless for suggesting the term ``daisies'' for the curves in Table 2.

\bigskip

\end{document}